\documentclass[11pt,leqno]{article}
\usepackage{amsthm,amsfonts,amssymb,epsfig,graphics,amsmath }
 \usepackage[latin1]{inputenc}\relax

\newcommand{\cA}{\mathcal{A}}
\newcommand{\cB}{\mathcal{B}}

 \newcommand{\diag}{{\rm diag}}

\newcommand{\RR}{{\mathbb R}}

\newcommand{\NN}{{\mathbb N}}
\newcommand{\ZZ}{{\mathbb Z}}

\newcommand{\CC }{{\mathbb C}}

\def\eps{\varepsilon}

\def\D{\partial }
 
\newcommand{\Id}{{\rm Id }}
\newcommand{\im}{{\rm Im }\, }
\newcommand{\re}{{\rm Re }\, }

\newcommand{\mez}{ \frac{1}{2} }

\newtheorem{theo}{Theorem}[section]
\newtheorem{prop}[theo]{Proposition}

\newtheorem{lem}[theo]{Lemma}
\newtheorem{defi}[theo]{Definition}
\newtheorem{ass}[theo]{Assumption}

\newtheorem{exam}[theo]{Example}
\newtheorem{rem}[theo]{Remark}

\numberwithin{equation}{section}
\title{{\bf Dispersive Stabilization }}

\author{\sc \small Guy M\'etivier\thanks{
IMB, Universit\'e de Bordeaux,
33405 Talence Cedex, France; metivier@math.u-bordeaux.fr.},  
\quad 
\sc \small Jeffrey Rauch\thanks{University of Michigan, Ann Arbor 48109 MI, USA  ; rauch@umich.edu;  Research partially supported by NSF under grant NSF DMS 0405899} 
}

\date{}

\begin{document}

%==================

\maketitle

\begin{abstract}
Ill posed linear and nonlinear initial value problems may  be stabilized, that it converted to 
to well posed initial value problems, by the addition of purely nonscalar linear dispersive terms. 
This is a stability analog of the Turing instability. This idea applies to systems of quasilinear Schr\"odinger equations from nonlinear optics.

\end{abstract} 

%===================

\section{Introduction}
 In nonlinear optics, one commonly encounters 
 coupled 
 systems of scalar Schr\"odinger equations
 \begin{equation}
 \label{eqa1.1}
  \D_t  u_j + i  \lambda_j \Delta_x u_j = \sum_{k=1}^N b_{j, k} (u,  \D_x) u_k  , \quad j \in \{ 1, \ldots, N \}, 
  \quad (t,x)\in \RR^{1+d},
 \end{equation}
where the $\lambda_j$ are real and the $b_{j, k}$ are first order
partial differential operators with coefficients depending smoothly
on $u$
(see \cite{Co} and the references therein). 
The nonlinear terms usually
depend on $u$ and $\overline u$,
 \begin{equation}
 \label{eqa1.2}
  \D_t  u_j + i  \lambda_j \Delta_x u_j = \sum_{k=1}^N c_{j, k} (u,  \D_x) u_k   
  +  d_{j, k} (u,  \D_x) \overline u_k ,
 \end{equation}
 where the  $c_{j,k}$ and $d_{j,k}$   are first order in $\D_x$. 
Introducing $u$ and $\overline u$ as  unknowns reduces to 
the form \eqref{eqa1.1} for a doubled real system. 

For the local in time
existence of smooth solutions, the easy case   is when the 
first order part, $B (u, \D_x) u$ on  the right hand side is symmetric.
In this symmetric
case there are easy $L^2$ estimates, followed by $H^s$ estimates  obtained
by commutations, which imply the local well-posedness of the Cauchy problem for
 \eqref{eqa1.1} in Sobolev spaces $H^s(\RR^d)$ for $s > 1 + \frac{d}{2}$.

In many applications, $B (u, \D_x)$ \emph{is not symmetric} and even more
$\D_t - B(u, \D_x)$ is not hyperbolic and   the Cauchy 
problem for  $\D_t u- B(u, \D_x) u = 0$ can be as  ill posed
as the Cauchy  problem for the Laplacian. 
However, the Cauchy problem
for \eqref{eqa1.1} may be well posed even if it is ill posed for the first order part.  This is so even though the dispersive terms
$i\lambda_j\Delta$ are not at all dissipative.  We call this
phenomenon 
{\it dispersive stabilization}. 

\vskip.2cm

\begin{exam}  
\label{ex:scrod}
With $x\in \RR$ the Cauchy 
problem for the system,
$$
\D_t u\ +\
 i \frac{\partial^2u}{\partial x^2} 
 \  +\
  \D_{x} v = 0, 
\qquad
\D_t v
\ -\
 i   \frac{\partial^2u}{\partial x^2} 
 \ -\
   \D_{x} u = 0,
$$
is well posed in $H^s$ even though the 
first order part defines a badly ill posed 
initial value problem.
\textup{This is  proved  by Fourier  transformation in 
$x$.  The amplification matrix is
$$
\exp\,
t\,\left(
\begin{matrix}i\xi^2 &-i\xi\cr
i\xi & -i\xi^2
\end{matrix}
\right)
$$
 For large $\xi$ the matrix has purely imaginary eigenvalues
 close to $\pm i\xi^2$ and is uniformly diagonalisable showing
 that the amplification matrix is uniformly  bounded for  $\xi\in \RR$
 and $t$ belonging to compact sets.   The bound grows exponentially
 in time.  The growth comes from $|\xi|\le R$.}
\end{exam}

The fact that the addition of a term ${\rm diag}\,(i\partial_x^2,-i\partial_x^2)$
whose evolution is neutrally stable can stabilize a stongly ill posed
Cauchy problem is not intuitively clear.  There are many related
results of this sort.  The simplest is the following assertion about
linear constant coefficient ordinary differential equaitons
in the plane.

\begin{exam}
If $A$ and $B$ are $2\times 2$ real matrices, knowing the 
stability origin as 
equilibrium of 
$$
X^\prime = A\,X,
\qquad
{\rm and},
\qquad
X^\prime = B\, X,
$$
one can draw no conclusion about the stability of 
 the equilibrium
 $X^\prime = (A+B) X$.
 \textup{The best know is the {\it Turing instability}
 \cite{tu}
 for which $A$ and $B$ have eigenvalues with strictly negative
 real part so the input dynamics are exponentially stable
 and the sum dynamics can be unstable.  
 Each of the stable dynamics is dissipative for
 certain scalar products.  When the scalar
 products are different the Turing instability is
 possible.
 One but not both
 of the matrices $A,B$ can be symmetric.  }
\end{exam}

\noindent
A related example is the two dimensional wave equation.

\begin{exam}
For the system version of the $2-d$ wave equation,
$$
u_{t} 
\ +\ 
\begin{pmatrix}
1&0\\-0&-1
\end{pmatrix}
u_x
\ +\
\begin{pmatrix}
0&1\\1&0
\end{pmatrix}
u_y
\ =\
0
$$
each of the split dynamics
$$
u_t  
\ +\ 
\begin{pmatrix}
1&0\\0&-1
\end{pmatrix}
u_x
\ =\ 0,
\qquad
u_t
\ +\
\begin{pmatrix}
0&1\\1&0
\end{pmatrix}
u_y
\ =\
0
$$
defines a bounded semigroup on 
$L^\infty(\RR^2)$.  The first
(resp. second) conserves
$$
\|u_1\|_{L^\infty},
\ \ 
{\rm and}
\ \ 
\|u_2\|_{L^\infty},
\qquad
\Big({\rm resp.}
\quad
\|u_1+u_2\|_{L^\infty},
\ \ 
{\rm and}
\ \ 
\|u_1-u_2\|_{L^\infty}
\Big)\,.
$$
The sum defines a dynamics so that the
map
$$
u(0,x,y)
\quad
\mapsto
\quad
u(t,x,y)
$$
is unbounded on $L^\infty(\RR^2)$ for all $t\ne 0$.
\end{exam}

 This analysis in this paper resembles example \ref{ex:scrod}.
 We do not use the local smoothing properties of 
 Schr\"odinger equations.
Instead, the 
Fourier  transform method
is extended using 
the paradifferential  
calculus. The idea is 
to conjugate $i A - B$ 
by a change of variable $I+V$
with $V$ of order $-1$ to a 
\emph{normal  form} 
\begin{equation}
\label{a1.3}
 (\Id  + V) (i A  -   B) (\Id + V)^{-1}  =   i A - \widetilde B    
\end{equation}
up to zero-th order terms, with $\widetilde B = i [V, A] - B $ symmetric. 
 The conjugation \eqref{a1.3} means that 
 the principal symbols satisfy 
\begin{equation}
\label{a1.4}
\sigma_{\widetilde B } =  \sigma_ B  + i  [ \sigma_A, \sigma_ V  ]      . 
\end{equation}
Equivalently,  the energy estimates are obtained using 
 the  \emph{pseudodifferential symmetrizers} 
\begin{equation}
\label{a1.5}
S  = \Id  +  V^*  + V 
\end{equation}
If the $\lambda_j$ are pairwise distinct, one can reduce $B$ to its 
diagonal part to prove the following result.

\begin{theo}
\label{tha1}  If the $\lambda_j$ are real and pairwise distinct and if the diagonal terms
$b_{j, j} (u, \D_x)$ have real coeficients,   then locally in time,
the Cauchy problem for
 \eqref{eqa1.1} is well posed in 
 the Sobolev spaces $H^s(\RR^d)$ for $s > 1 + \frac{d}{2}$. 
\end{theo}

\noindent
An analogous result 
for the systems \eqref{eqa1.2} is the following.

\begin{theo}
\label{tha1b}  Suppose that

-  the  $\lambda_j$ are real and pairwise distinct

- the diagonal terms  $b_{j, j} (u, \D_x)$ have real coeficients, 

- $c_{j, k}(u, \D_x)  = c_{k, j}(u, \D_x)$ for all pairs $(j,k)$ such that 
$\lambda_j +\lambda_k =0$.

\noindent  Then locally in time, the Cauchy problem for
 \eqref{eqa1.2} is well posed in the
 Sobolev spaces $H^s(\RR^d)$ with  $s > 1 + \frac{d}{2}$. 
\end{theo}

 In the next section we give a more general statement which allows for
more general nondiagonal second order terms.
In particular the $\lambda_j \Delta_x$  can be replaced by different second order elliptic operators
$A_j (\D_x)$. 
The idea of using pseudodifferential symmetrizers is 
related to  the proof in \cite{Co} where the symmetry is obtained after 
differentiation of the equations and clever linear recombination.
This amounts 
to using differential symmetrizers. 
Our  analysis    is 
a  systematic exploration of the idea. 
Because of the quasilinear
character of the equations, we use the paradifferential calculus
in place of the classical pseudodifferential version.   The latter
would have sufficed to treat semilinear analogues.
The paradifferential methods
can also be used to treat the 
stongly nonlinear case  $F (u, \D_xu)$
since such a term is reduced to a quasilinear
term using the paralinearization,
see Section \ref{sec:statement}.

For the systems case the dispersive
terms rotating at different speeds
regularize an explosive first order term.
For the scalar case, that is $N = 1$, 
such a stabilisation is not possible.
The  Cauchy problem for 
$\D_t - i \Delta_x  + i  \D_{x_1}$  is ill posed.
However, 
 if $\im b(x)$ satisfies suitable decay assumptions at infinity, 
then the Cauchy problem for  $\D_t - i \Delta_x  +  b(x) \cdot \nabla_x $  is well 
posed (see \cite{Mizohata}).  
Intuitively, the waves propagate to the regions where
$b$ is small and are no longer amplified.  The proofs 
use the dispersive and  local smoothing  properties 
of Schr\"odinger equations. This idea has been extensively studied.
Some of the foundational papers are
\cite{Mizohata2}, \cite{kato},
\cite{cs}, \cite{KPV},  \cite{KPV2}, and,
 references therein. 
It would be natural to combine such ideas with those
of dispersive stabilization with the goal of extending
the local existence to the case where the antisymmetric
part of 
$\widetilde B$ has suitable decay at infinity 
rather than requiring that it vanish.
We do not pursue this line of inquiry.

%%%%%%%%%%%%%%%%%%%%%%%%%%%%%%%%%%%%%%%%%%

\section{Statement of the result}
\label{sec:statement}
Consider the general equations,
\begin{equation}
 \label{eq1.1}
  \D_t u    +  i  A (\D_x) u   +   B(t, x,  u, \D_x ) u  =  0  ,
\end{equation}
with  $A$ second order and $B$ first order,
\begin{equation}
A(\D_x) = \sum_{j, k= 1}^d   A_{j,k} \D_{x_j} \D_{x_k}, 
\end{equation}
\begin{equation}
B(t, x, u,  \D_x) = \sum_{j =  1}^d   B_j (t, x, u)  \D_{x_j}.  
\end{equation}
The matrices  $ B_j (t, x, u) $   are assumed to be $C^\infty$ functions of  
$ (t, x, \re u, \im  u)$, 
so that for each $\alpha$ and bouded $K\subset\CC^N$,
$$
\partial_{t,x, \re u, \im  u}^\alpha B\ \in\
L^\infty([0, T] \times \RR^d \times K).
$$

\noindent
With the example \eqref{eqa1.1} in mind, we assume that $A$ is smoothly block-diagonalizable: 

\begin{ass}
\label{ass1}
 For all  $\xi \in \RR^n\backslash\{0\} $,   $A (\xi) = \sum A_{j,k} \xi_j \xi_k$  is  self-adjoint. 
 Moreover, there are smooth real eigenvalues  $\lambda_p(\xi)$  and smooth
 self-adjoint  eigenprojectors $\Pi_p (\xi)$ such that 
 $$
 A(\xi)  = \sum_p \lambda_p(\xi ) \Pi_p (\xi) . 
 $$
\end{ass}

 This assumption is satisfied if $A$ is self-adjoint with   eigenvalues of  constant multiplicity. The assumption allows for regular crossing of eigenvalues. 
The  conditions on $B$ involve,
$$
\im B \ :=\
 \frac{1}{2 i} (B - B^*)  .
$$

\begin{ass}
\label{ass2}    For all $p  $ and  $ q$, 
$$
\lambda_p(\xi)= \lambda_q (\xi)
\quad
\Longrightarrow
\quad
\Pi_p (\xi) \,\im  B(t, x, u, \xi)\, \Pi_q (\xi) =0\,.
$$
In addition,  there are smooth matrix valued functions 
$V_{p, q} (t, x, u, \xi)$ so that 
\begin{equation}
\label{2.4} 
\Pi_p (\xi)\, \im  B(t, x, u, \xi)\, \Pi_q (\xi) 
\  =\
   \big( \lambda_p(\xi) -  \lambda_q (\xi)\big) V_{p, q} (t, x, u, \xi). 
\end{equation}  
\end{ass}

\begin{rem}
\label{rem11}  
\textup{The condition \eqref{2.4} is 
automatic where $\lambda_p(\xi) \ne \lambda_q(\xi)$, it defines 
$V_{p, q}$.  Assumption \eqref{ass2} contains two types of information.}
 
\quad $\bullet$  \textup{For any $\underline \xi$, if $\underline \lambda$ is an eigenvalue of 
$A(\underline \xi)$ and $\Pi(\underline \xi)$  the spectral projector, then 
 $\Pi (\underline \xi)  B (t, x, u, \underline \xi)  \Pi(\underline \xi)$ is self adjoint.  
 If the eigenvalue remains of constant multiplicity 
 for $\xi$ near $\underline \xi$, nothing more needs to be added for this polarization. In particular,  
if all the distinct eigenvalues $\lambda_p(\xi)$  of $A(\xi) $ have constant multiplicity, 
the   Assumption~ \ref{ass2}  reduces to the condition that  the matrices 
 $\Pi_p (\xi)  B(t, x, u, \xi) \Pi_p (\xi)$ are  self-adjoint.}
 
 \quad $\bullet$  \textup{If the eigenvalue $\underline \lambda $ splits into several eigenvalues
$ \lambda_p(\xi)$   for $\xi$ near $\underline \xi$, the condition~\eqref{2.4} means that 
not only $\Pi_p (\xi) \im  B(t, x, u, \xi) \Pi_q (\xi)  $ vanishes at $\underline \xi$ 
and on the variety $\{\lambda_p = \lambda_q\}$,  but also 
that  $  \lambda_p(\xi) -  \lambda_q (\xi) $ is a divisor. 
 In particular, if $\widetilde \Pi (\xi)$ denotes the spectral projector on the invariant space associated 
to the eigenvalues close to $\underline \lambda $, this condition is 
locally satisfied 
with $V_{p,q}=0$ 
whenever $\widetilde \Pi (  \xi)\,  B (t, x, u,  \xi)\,  \widetilde\Pi(  \xi)$ is self-adjoint. 
This is so since 
$$
0= \widetilde\Pi \,{\rm Im}\, B\,\widetilde\Pi =
\sum_{p,q} \Pi_p {\rm Im}\, B\Pi_q,
\quad
{\rm so},
\quad
\Pi_p \,{\rm Im}\, B\,\Pi_q = \Pi_p \,\widetilde\Pi\, {\rm Im}\, B\,\widetilde\Pi \,\Pi_q=0.
$$}
\end{rem}

\begin{rem}
  \textup{ There is no assumption on the spectrum of $B(t, x, u, \xi)$.
  In particular, 
$\D_t + B$ may be nonhyperbolic and thus strongly  unstable in Hadamard's sense. 
The dispersive term $A$ has a stabilizing effect, provided that the
condition in  Assumption~\ref{ass2} is satisfied. 
For this reason models of this type appear often in the descriptions
of instabilities, for example that of Raman.  The dispersive
stabilisaton regularizes  to 
a well posed causal model albeit with the possibility of 
growth for moderate wave numbers as in the example. }
\end{rem}

We show that under the Assumptions~\ref{ass1} and \ref{ass2} the Cauchy problem for 
\eqref{eq1.1} is well posed in $H^s$ for $s > \frac{d}{2} + 1$, locally in time.

\begin{theo}
\label{th2.5}  If  Assumptions~$\ref{ass1}$ and $\ref{ass2}$
hold,  $s > \frac{d}{2} + 1$,   and,  
$h \in H^{ s}(\RR^d)$,  there is $T > 0$ and a unique solution
$ u \in C^0([0, T]; H^s(\RR^d))$ of    \eqref{eq1.1} with 
$u|_{t = 0} = h$. 
\end{theo} 

\begin{exam}[From \cite{Co}] 
  \textup{ $A$ is block diagonal  $ A = \diag \{ \lambda_p  \Id_p  \}  $ with  
real $\lambda_p(\xi)$  homogeneous of degree two and $\lambda_p (\xi) \ne \lambda_q(\xi)$
for $p \ne q$ and $\xi \ne 0$. The second assumption is trivially satisfied if the diagonal blocks
$B^{p,p}$ vanish.}
\end{exam}

For applications, it is interesting to  make explicit the 
assumptions when 
the first order part depends on $\overline u $,
\begin{equation}
 \label{eq1.5}
  \D_t u    +  i  A (\D_x) u   +   B(t, x, u, \D_x ) u  + 
 C(t, x,  u, \D_x ) \overline u  =  0  
\end{equation}
Introducing $ v = \overline u$ as a variable and setting 
$U = {}^t ( u,  v)$, the equation reads: 
\begin{equation}
 \label{eq1.6}
  \D_t  U     +  i  \cA (\D_x) U   +   \cB(t, x,  u , \D_x ) U   =  0  
\end{equation}
 with 
 \begin{equation}
 \cA  = \begin{pmatrix} A (\D_x) & 0 \\ \\0 & - A(\D_x)  \end{pmatrix}, 
 \qquad 
  \cB  = \begin{pmatrix} B  &  C  \\   \\
  \overline C  & \overline B  \end{pmatrix}.  
 \end{equation}
In this context, the Assumption~\ref{ass2} becomes the following.
\begin{ass}
\label{ass3}    For all $p  $ and  $ q$,  $\Pi_p (\xi) \im  B(t, x, u, \xi) \Pi_q (\xi) $ vanishes when 
$\lambda_p(\xi)= \lambda_q (\xi)$ and 
 $\Pi_p (\xi) \big( C(t, x, u, \xi) -  {}^t  C(t, x, u, \xi)  \big) \Pi_q (\xi) $ vanishes when 
$\lambda_p(\xi)+  \lambda_q (\xi) = 0 $. 
In addition, there are smooth matrices  
$V_{p, q} (t, x, u, \xi)$ and $W_{p, q} (t, x, u, \xi)$ such that 
\begin{eqnarray}
\label{2.8} 
\Pi_p  ( \im  B ) \Pi_q     =   ( \lambda_p  -  \lambda_q  ) V_{p, q} ,
\\ 
\label{2.9}
\Pi_p    \big( C  -  {}^t  C   \big) \Pi_q     =   ( \lambda_p + \lambda_q  ) W_{p, q}  .
\end{eqnarray} 
\end{ass}

\begin{theo} 
\label{th2.8} Under Assumptions~$\ref{ass1}$ and $\ref{ass3}$, for $s > \frac{d}{2} + 1$   and   
$h \in H^{ s}(\RR^d)$,  there is $T > 0$ and a unique solution
$ u \in C^0([0, T]; H^s(\RR^d))$ of    \eqref{eq1.5} with 
data $u|_{t = 0} = h$. 
\end{theo}

We briefly discuss the case of equations with fully nonlinear right hand side,
\begin{equation}
 \label{eq1.3}
  \D_t u    +  i  A (\D_x) u   +   F(t, x, u, \D_x  u)  =  0,    
\end{equation}
where $F(t, x, u, v_1, \ldots, v_d)$  is a smooth function of 
$(t, x, \re u, \im u)$ and  of $(\re v_1, \ldots, \im v_d)$.  
Our analysis relies 
on a paralinearization of the first order term, so that the analogues  
of    $B$ and $C$ are
\begin{eqnarray}
B (t, x, u, v, \xi ) = \sum_{j} \xi_j \frac{\D F}{\D v_j} (t, x, u, v) 
\\
C (t, x, u, v, \xi ) = \sum_{j} \xi_j \frac{\D F}{\D \overline v_j} (t, x, u,  v) 
\end{eqnarray}
with 
$$
\frac{\D  }{\D v_j} = \mez \frac{\D }{\D \re v_j} - \frac{i}{2} \frac{\D }{\D \im v_j}, 
\qquad 
\frac{\D  }{\D \overline v_j} = \mez \frac{\D }{\D \re v_j} + \frac{i}{2} \frac{\D }{\D \im v_j}
$$
as usual. The new condition is that \eqref{2.8} \eqref{2.9} are satisfied 
with smooth matrices $V_{p, q} (t, x, u, v) $ and  $W_{p, q} (t, x, u, v) $. 
In this case, the Cauchy problem is well posed in $H^s$ 
for $s > \frac{d}{2} + 2$.

%%%%%%%%%%%%%%%%%%%%%%%%%%%%%%%%%%%%%%%%
%%%%%%%%%%%%%%%%%%%%%%%%%%%%%%%%%%%%%%%%

\section{Basic $\mathbf  L^2$ estimate }
\label{sec3}

We solve \eqref{eq1.1}  by Picard iteration.  
Consider first the linear 
problem,
\begin{equation}
 \label{eq3.1}
  \D_t u    +  i  A (\D_x) u   +   B(t, x,  a, \D_x ) u  =  f   , \qquad  u_{| t = 0}  = h, 
\end{equation}
where 
\begin{equation}
\label{3.2} 
a \in C_w^0( [0, T]; H^s(\RR^d) ),  \quad  \D_t a \in C_w^0( [0, T]; H^{s-2}(\RR^d) ) 
\end{equation} 
with  $s > \frac{d}{2} + 1$ and $C^0_w([0, T]; H^\sigma)$ denotes the space of 
functions which are   continuous from $[0, T]$ to $H^\sigma$ equipped with the 
weak topology. 

\begin{theo}
\label{th3.1} 
There are functions $C_0$  and $C_1$    so
that the solution of \eqref{eq3.1} satisfies 
\begin{equation}
\label{3.3}
\big\| u  (t)   \big\|_{L^2}  \le  C_0( K_0) e^{ t C_1 (K_1)}
\Big(  \big\| u  (0)   \big\|_{L^2}   + 
\int_0^t   \big\| f  (t')   \big\|_{L^2}  dt'\Big) 
\end{equation}
with 
\begin{eqnarray}
\label{3.4}
&&K_0 := \| a \|_{L^\infty ([0, T]\times \RR^d) } ,
\\
\label{3.5}
&&K_1 : = \| a \|_{L^\infty ([0, T]; H^s( \RR^d) ) }  +  \| \D_t a \|_{L^\infty ([0, T]; H^{s-2}( \RR^d) ) }. 
\end{eqnarray}

\end{theo}

\begin{lem} [Conjugation]  
\label{lem32}
For $| \xi | $ large, there is a smooth
invertible  matrix 
$ V_{-1}  (t, x, u, \xi) $, 
homogeneous of degree $-1$ in $\xi$, such that 
\begin{equation}
\label{conjug}
      B( t, x, u, \xi)  -  [  V_{-1} (t, x, u, \xi)  ,    A (\xi)  ]   
\end{equation} 
is self adjoint  and homogeneous of degree $1$ in $\xi$. 
\end{lem}

\begin{proof}
 Set 
 $$
 V_{-1}  : = \sum_{ p  ,  q} \frac {1}{ \lambda_p - \lambda_q} \Pi_p ( \im B )  \Pi_p 
 $$
 so that  $B - [ V_{-1}, A ] = \sum \Pi_p B \Pi_p $ is self adjoint.
 \end{proof}

\begin{proof}[Proof of Theorem~$\ref{th3.1}$.]
  Use the paradifferential calculus and the notations of Section~5.  
  
 {\bf a) }   For simplicity  denote by $B_j (t, x)  $ the matrix $B_j(t, x, a(t, x)) $ 
 and by $B = B(t, x, a(t,x) ,\xi)$ the symbol $  \sum \xi_j B_j$.  Because $s > 1 + \frac{d}{2}$, \eqref{3.2} implies that 
 $B_j \in C^0 ([0, T]; H^s)$, $\D_t B_j \in C^0([0, T]; H^{s-2})$ and  
 \begin{equation}
 \label{3.7}
   \| B_j \|_{L^\infty ([0, T]; H^s( \RR^d) ) }  +  \| \D_t B_j \|_{L^\infty ([0, T]; H^{s-2}( \RR^d) ) }
   \le C_1(K_1).  
   \end{equation}
  In particular, as a symbol, $B $ belongs to the class $ \widetilde \Gamma^1_1$ 
  introduced in Definition~\ref{def31b}.  Using the  paralinearization Proposition~\ref{paralin}
  we see that 
  $f_1 := B(t, x, \D_x) u - T_{i B} u  $ satisfies
  \begin{equation}
  \label{3.8}
  \| f_1(t) \|_{L^2}  \le C_1 (K_1) \| u(t) \|_{L^2} ,  
  \end{equation}
  and $u$ satisfies   the paralinearized equation: 
  \begin{equation}
 \label{3.9}
  \D_t u    +  i  A (\D_x) u   +   T_{i B}  u  =  f + f_1   , \qquad  u_{| t = 0}  = h, 
\end{equation}
  
  {\bf b)} Similarly, use the simplified notation 
  $  V(t, x, \xi) = V_{-1} ( t, x, a(t, x), \xi) \zeta (\xi) $ where $\zeta \in C^\infty (\RR^d)$ vanishes 
  near the origin and is equal to 1 for $| \xi | \ge 1$. 
Note that $V \in \widetilde \Gamma^{-1}_1$ and that for all $\alpha$ there are functions
$C_{0, \alpha}$ and $C_{1, \alpha} $ such that for all $t \in [0, T]$ and $\xi \in \RR^d$. 
  \begin{eqnarray}
  \label{3.10}  
&&  \big\| 
 \partial_{\xi  }^\alpha   V  ( t, \cdot  , \xi ) \big\Vert_{ L^\infty  }    \le  C_{0, \alpha}  (K_0)  (1 +   \vert  \xi  \vert )   ^ {    \vert
\mu - \alpha \vert     }
\\
\label{3.11}  
 && \big\| 
 \partial_{\xi  }^\alpha \D_t V  ( t, \cdot  , \xi ) \big\Vert_{ H^{s-2} }    \le  C_{1, \alpha}  (K_1)  (1 +   \vert  \xi  \vert )   ^ {    \vert
\mu - \alpha \vert     }
        \, .
\end{eqnarray}
 
Use  a  symmetrizer,
  \begin{equation}
  \label{3.12}
  \Sigma \ :=\ \Id  +   i T_{V } -  i  (T_V)^*   + \gamma (1- \Delta_x)^{-1} . 
  \end{equation}
By Proposition~\ref{action} and Remark~\ref{rem56}, there is a constant $C_0(K_0)$ 
which depends only on $K_0$ such that
  $$
 \big\| T_V u  (t)  \big\|_{ H^{1} }  \le  C_0(K_0)  \| u (t) \big\|_{L^2 }. 
  $$
  Therefore,
  $$
  \big(   \Sigma u, u \big)_{L^2}  \ge  \big\| u \big\|^2_{L^2 } -
    2 C_0 (K_0)  \big\| u \big\|_{L^2 } \big\| u \big\|_{H^{-1} } + \gamma  \big\| u \big\|^2_{H^{-1} }.
  $$
Choose  $\gamma = \gamma(K_0)$  so that 
  \begin{equation}
  \label{3.13}
\big(   \Sigma u, u \big)_{L^2} \ \ge\ \mez \big\| u \big\|^2_{L^2 }. 
  \end{equation}
Then, with  another constant $C_0(K_0)$, 
  \begin{equation}
  \label{3.13b}
  \big\| \Sigma u(t) \big\|_{L^2} \le C_0 (K_0) \big\| u(t) \big\|_{L^2}. 
  \end{equation} 
  
  {\bf c) } Compute 
\begin{equation}
\label{3.14}
  \frac{d}{dt} \big(   \Sigma(t)  u (t) , u(t)  \big)_{L^2}    = 2 \re 
  \big(   \Sigma \D_t  u, u \big)_{L^2}    + \big(  [\D_t ,  \Sigma]  u, u \big)_{L^2}  . 
 \end{equation}
    By Lemma~\ref{lememb} and Proposition~\ref{commdt},  $ [\D_t ,  \Sigma] =  [\D_t , T_V ]  + [\D_t, T_V]^*$ is bounded 
    from $L^2$ to $L^2$ and,
   \begin{equation}
   \label{3.15}
    \big( [\D_t, \Sigma ] u  (t) ,  u(t)  \big)_{ L^2 }  \le  C_1(K_1)   \| u (t) \big\|^2_{L^2 }\,. 
   \end{equation}
   Next, observe that 
 $  T_V A(\D_x)  =  A(D_x ) T_V +  [T_V, A(\D_x)]  = - T_{V A(\xi) }  $. Therefore,  
    the equation and the symbolic calculus of Proposition~\ref{symbcal} imply that 
    $$
    \Sigma \D_t  u =  - A(\D_x) u    -   i  A(\D_x) T_V A(\D_x)  +   i  ( T_V) ^{*} A(D_x) 
     -      i  T_{ \widetilde B  } u  + \Sigma f +  f_2 
    $$
    where $\widetilde B $ is the symbol $B(t, x, \xi) - [V(t, x, \xi), A(\xi)] \in \widetilde \Gamma^1_1$ 
    and $f_2$ satisfies an estimate similar to \eqref{3.8}. 
    By Lemma~\ref{lem32}  $\widetilde B$  is self adjoint for $| \xi | \ge 2$, and hence Proposition~\ref{adjoint} implies that 
    $$
    \re \big(  i T_B u (t) , u (t)  \big)_{L^2}  \le  C_1(K_1)   \| u (t) \big\|^2_{L^2 }\,.
    $$
    Since $A(\D_x)$ is self adjoint, we conclude that 
    \begin{equation}
    \label{3.16}
      \frac{d}{dt} \big(   \Sigma(t)  u (t) , u(t)  \big)_{L^2}    
      \le  2  \|\Sigma  f (t) \big\|_{L^2 }   \| u (t) \big\|_{L^2 }   +     C_1(K_1)   \| u (t) \big\|^2_{L^2 }\,.  
    \end{equation} 
    Equations \eqref{3.13} and \eqref{3.13b}  imply   estimate \eqref{3.3}.  
\end{proof} 
\medbreak

%%%%%%%%%%%%%%%%%%%%%%%%%%%%%%%%%
%%%%%%%%%%%%%%%%%%%%%%%%%%%%%%%%%

\section{Sobolev estimates and nonlinear existence}

  $H^s$ estimates for the linearized equation \eqref{eq3.1} are obtained by differentiating 
  the equation.  The commutators $ [\D_x^\alpha, B (t, x, a , \D_x] u$ are estimated by
  standard nonlinear estimates as in the analysis of first order hyperbolic equations.
  Because $s > \frac{d}{2} + 1$, for $|\alpha | \le s$, one has,
  \begin{equation}
  \label{4.1}
  \Big\|\, \big[\D_x^\alpha\,,\, B (t, x, a , \D_x)\big] u (t)   \Big\|_{L^2}  \le  C_1(K_1) \big\| u  (t)   \big\|_{H^s} 
  \end{equation} 
 This implies the following estimates.

 \begin{prop}
\label{prop41} 
There are functions $C_0$  and $C_1$    such  that smooth solutions of \eqref{eq3.1} satisfy 
\begin{equation}
\label{4.2}
\big\| u  (t)   \big\|_{H^s} 
\  \le\
  C_0( K_0) e^{ t C_1 (K_1)}
\Big(  \big\| u  (0)   \big\|_{H^s}   + 
\int_0^t   \big\| f  (t')   \big\|_{H^s}  dt'\Big) 
\end{equation}
with  $K_0$ and $K_1$ defined at \eqref{3.4} and \eqref{3.5}. 
 
\end{prop} 

As in the hyperbolic theory, this estimates implies the 
following strong continuity result.

\begin{prop}
\label{prop4.2}
Suppose that $a $ satisfies \eqref{3.2}, $f \in L^1 ([0, T], H^s)$ and 
$h\in H^s$. If $u \in C^0_w ([0, T]; H^s)$ is a solution of 
\eqref{eq3.1}, then $u \in C^0([0, T], H^s)$. 
\end{prop} 

\begin{proof}
With  $J_\eps = (1 -\eps \Delta_x)^{-1} $, one checks that 
$J_\eps u $ satisfies 
\begin{equation}
 \label{eq3.1smooth}
  \D_t J_\eps u    +  i  A (\D_x) J_\eps u   +   B(t, x,  a, \D_x ) J_\eps u  =  f_\eps   , 
  \qquad  J_\eps u_{| t = 0}  =  J_\eps h, 
\end{equation}
with $f_\eps \to f$ in $ L^1 ([0, T], H^s)$. Applying the estimates to $J_\eps u$, one 
obtains that $J_\eps u$ is a Cauchy family in $C^0([0, T], H^s)$, implying that 
$u\in C^0([0, T], H^s)$.
\end{proof} 

\noindent
Turn to the   proof of the   main result. 
More details can be found in \cite{Met2}.

\begin{proof}[Proof of Theorem~$\ref{th2.5}$]
(i)  To solve  \eqref{eq3.1}  for  $a$ satisfying 
\eqref{3.2} use the mollified equations 
\begin{equation}
 \label{eq3.1mol}
  \D_t u^\eps     +  i  A (\D_x) J_\eps u^\eps    +   B(t, x,  a, \D_x ) J_\eps u^\eps  =  f   , \qquad  u_{| t = 0}  = h, 
\end{equation}
where $J_\eps = (1 -\eps \Delta_x)^{-1} $. For fixed  $\eps$, this is a linear o.d.e in $H^s$
since $A(D_x)J_\eps $ and $B J_\eps$ are bounded. One checks that the proof 
of the estimates \eqref{4.2} for the solutions of \eqref{eq3.1} immediately extends 
to the solutions of \eqref{eq3.1mol},  because $\{J_\eps\}$ is a bounded 
family of pseudodifferential operators of degree $0$,  are the new commutators 
they generate are remainders in the symbolic calculus developed in section~\ref{sec3}. 
Therefore, the $u^\eps$ are uniformly bounded in 
$C^0 ([0, T] ; H^s)$.
The equation shows that they are bounded
in $ C^1([0, T], H^{s-2})$. 

Extracting  a subsequence and passing to the 
weak limit yields a  solution  $u \in C^0_w ([0, T], H^s)$. 
By Proposition~\ref{prop4.2},   $u \in C^0 ([0, T], H^s)$.

(ii)  Solve the nonlinear equation using the iteration scheme, 
\begin{equation}
 \label{eq3.1it}
  \D_t u_{n+1}     +  i  A (\D_x) u_{n+1}   +   B(t, x,  u_n \D_x ) u_{n+1}  = 0 , 
  \qquad  u_{n+1}{}_{| t = 0}  = h. 
\end{equation}
Using the estimate \eqref{4.2}, one proves that there is $T> 0$ such that 
the sequence $\{u_n\}$ is bounded in  $C^0 ([0, T], H^s)$ and in 
$ C^1 ([0, T], H^{s-2})$. Knowing this bound in high norm, 
one checks that the sequence $u_n$ converges in a low norm 
$C^0([0, T]; L^2)$. Passing to the limit gives a solution of 
\eqref{eq1.1} $u \in C^0_w ([0, T], H^s)$, which also belongs to 
$ C^1 ([0, T], H^{s-2})$. Using Proposition~ \ref{prop4.2}, 
one obtains that $u\in  C^0 ([0, T], H^s)$.
\end{proof}

%%%%%%%%%%%%%%%%%%%%%%%%%%%%%%%%%
%%%%%%%%%%%%%%%%%%%%%%%%%%%%%%%%%
%%%%%%%%%%%%%%%%%%%%%%%%%%%%%%%%%

%%%%%%%%%%%%%%%%%%%%%%%%%%%%%%
 
\section{Handbook of paradifferential calculus}

The symmetrizers
are paradifferential  operators
in the variables $x$, depending on the parameter $t$.
This section  reviews the paradifferential
calculus  extended to the case of time dependent 
symbols.

 \subsection{The spatial calculus}

 Consider operators on $\RR^d$. The variables are denoted
$x $ and the frequency variables
$\xi $.

\begin{defi} [Symbols]  Let $\mu \in \RR$.
\label{def31}

i) $\Gamma^\mu_0$ denotes the space of locally $L^\infty$ functions
$a(x, \xi) $ on $\RR^d \times \RR^{d}  $
which are $C^\infty$ with respect to $\xi$ and such that for all
$\alpha \in \NN^d$ there is a constant $C_\alpha$ such that
\begin{equation}
\label{estimsymb3}
\forall (x, \xi) \, , \quad
\vert \partial_{\xi  }^\alpha a ( x, \xi ) \vert \, \le \,
C_\alpha \,
        (1 +   \vert  \xi  \vert )   ^ {  \mu -  \vert
\alpha \vert   }  \, .
\end{equation}

ii)  $\Gamma^\mu_1$ denotes the space of symbols $a \in \Gamma^\mu_0$
such that for all $j$, $\partial_{x _j} a \in \Gamma^\mu_0$.

\end{defi}

\bigbreak

The paradifferential calculus in $\RR^d$, was introduced
 by J.M.Bony \cite{Bony} (see also
\cite{Meyer}, \cite{Hormander}, \cite{Taylor},
\cite{MetKochel}).
The reference  \cite{Met2}
gives a detailed account of the time dependent 
results needed here. The calculus 
associates operators $T_a$ to   symbols 
$a \in \Gamma^\mu_0$.
They act in the scale of Sobolev spaces $H^s(\RR^d)$. Moreover, there is a symbolic calculus
at order one for symbols in $\Gamma^\mu_1$. 
Recall here the definition, as we will need it later on. 

Consider a $C^\infty$ function $\psi(\eta, \xi )$ 
 on  $\RR^n \times \RR^n $ such that 

\quad 1) there are $\varepsilon_1 $ and $\varepsilon_2$ such that
$0< \varepsilon_1 < \varepsilon_2 < 1$ and 
\begin{equation}
\label{p5.2}
\left\{\begin{aligned}
& \psi(\eta, \xi ) \, = 1 \quad {\rm for } \  
\vert \eta \vert \le 
\varepsilon_1 ( 1  + \vert \xi \vert )  \,  \,
\\
& \psi(\eta, \xi) \, = 0 \quad {\rm for } \  
\vert \eta \vert \ge \varepsilon_2 
(1 + \vert \xi \vert ) \,   \, . 
\end{aligned}\right.
\end{equation} 
 
\quad 2) for all $(\alpha, \beta) \in \NN^n \times \NN^n$, there is 
$C_{\alpha,\beta}$ such that 
\begin{equation}
\label{p5.3}
\forall (\eta, \xi, \gamma)\,  : \quad
\vert \partial _\eta^\alpha
\partial _\xi^\beta \psi(\eta , \xi, \gamma) \vert  
\le \, C_{\alpha, \beta} 
 (1 + \vert \xi \vert)^{- \vert \alpha \vert - \vert \beta \vert }
\,. 
\end{equation} 
For instance one can  consider with $N \ge 3$: 
\begin{equation}
\label{p5.4} 
\psi_N(\eta, \xi ) \, = \, \sum_{k=0}^{+ \infty}  
\chi_{k-N} ( \eta  )  \varphi_k   (\xi) 
\end{equation}
where $\chi \in C^\infty_0(\RR^d)$ satisfies $0 \le \chi \le 1$ and
\begin{equation}
\label{Achi}
\chi(\xi  ) \, = 1  \quad {\rm for } \
\vert \xi  \vert \, \le 1.1 \, ,
\quad
\chi(\xi  ) \, = 0 \quad {\rm for } \
\vert \xi \vert \, \ge 1.9 \,, 
\end{equation}
and for $k \in \ZZ$, 
\begin{equation}
\chi_k(\xi)  = \chi\big(2^{-k}   \xi  \big) 
\end{equation}
 and 
\begin{equation}
\label{p5.7}
 \varphi_0  = \chi_0 
   \quad \mathrm{and\ for \ }  k \ge 1 \quad
\varphi _{k}   = \chi _{k} - \chi_{k-1}  .
\end{equation}

  A  function $\psi$ satisfying  \eqref{p5.2} \eqref{p5.3} is  an {\sl admissible cut-off}. 
Consider next $G^\psi (\, \cdot  \, , \xi)$ the inverse Fourier
transform of $\psi(\, \cdot  \, , \xi)$.   
For $ a \in \Gamma^\mu_0$ define 
\begin{equation} \sigma^\psi_a (x, \xi) \, := \, 
\int G^\psi(x - y, \xi) \, a(y, \xi )\,  dy 
\end{equation}
 or equivalently on  the Fourier side in $x$,
\begin{equation}
\widehat \sigma^\psi_a (\eta, \xi ) \, = \, 
\psi(\eta, \xi ) \, \widehat a(\eta, \xi ). 
\end{equation}
 The   symbol
$\sigma \in \Gamma^\mu_0$  and   belongs to   H\"ormander's class 
 $S^\mu_{1, 1}$.
 The paradifferential   operator $T^\psi_a$ is defined by  
\begin{equation}
\label{p5.10}
T^{\psi}_a  u (x) \, := \, \frac{1}{(2\pi)^n} 
\int e^{ i \xi \cdot x}\, \sigma_a^\psi (x, \xi) \, \widehat u(\xi) \, d
\xi\, . 
\end{equation}

  We collect here the main results.

\begin{prop} [Action]
\label{action}  
\label{actionpara} Suppose that $\psi$ is an admissible cut-off.

 i) When $a(\xi )$ is a symbol independent
of ${x} $, the operator $T^\psi_a$ is equal to the
Fourier multiplier $a(D)$.

 ii) For all $a \in \Gamma^\mu_{0}$ and $s\in \RR$,   
$T^\psi_a $ is a bounded operator from $ H^s(\RR^d)$ to 
$H^{ s-\mu}(\RR^d)$. 
\end{prop}

\medbreak
\begin{prop}
If   $\psi_1$ and $\psi_2$ are two  admissible cut-off, then 
for all $a \in \Gamma^\mu_{0}$ and $s\in \RR$,   
$T^{\psi_1}_a - T^{\psi_2}_a  $ is a bounded operator from $ H^s(\RR^d)$ to 
$H^{ s-\mu +1}(\RR^d)$. 
\end{prop} 

\begin{rem}
\textup{This proposition implies   that the choice of $\psi $ is essentially
irrelevant in our analysis, as in
\cite{Bony}. To simplify notation, make a definite choice of $\psi$,
for instance $\psi =\psi_N$ with 
$N = 3$ as in \eqref{p5.4} and  use the notation $T_a$   for $T^\psi_a$. }
\end{rem}

 \begin{prop} [Symbolic calculus]
\label{symbcal}
Consider  $a \in \Gamma^\mu_{1} $ and $b \in \Gamma^{\mu'}_{1}$.
Then $a b \in \Gamma^{\mu+\mu'}_{1}$ and for all $s \in \RR$, 
$ T_a \circ T_b - T_{a b} $    
is   bounded   from $ H^s(\RR^d)$ to 
$H^{ s-\mu- \mu' +1}(\RR^d)$. 

If $b$ is independent of $x  $, then
$ T_a \circ T_b = T_{a b}$\, .
\end{prop}
\noindent
These results  extend  to matrix valued symbols and operators.

\medbreak
\begin{prop} [Adjoints]
\label{adjoint}
Consider a matrix valued symbol $a \in \Gamma^\mu_{1}$. Denote by
$(T_a) ^*$ the adjoint operator of $T_a$ in
$L^2(\RR^{d})$  and
by $a^*(x, \xi) $ the adjoint of the matrix $a (x, \xi)$.  Then
$ (T_a) ^* - T_{a^*}  $   is  bounded  from $ H^s(\RR^d)$ to 
$H^{ s-\mu +1}(\RR^d)$. 
\end{prop}

\begin{rem}
\label{rem56}
 \textup{The norm of the operators acting in the indicated Sobolev spaces
 are uniformly bounded when the symbols $a$ and $b$ 
 belong to  bounded subsets of the   symbol classes. }
 \end{rem}

Bounded functions of $x  $ are particular examples of
symbols in the class $\Gamma^0_0$, independent of the frequency variables
$\zeta $. In this case, $T_a$ is called
a {\it paraproduct} in \cite{Bony}.

\begin{prop} [Paralinearization]
  \label{paralin}
 There is a constant $C$
such that for all $a \in W^{1, \infty}  $ and all
$u \in L^2(\RR^d)   $
$$
\begin{aligned}
  \big\| a \D_{x_j} u  - T_a \D_{x_j} u \big\|_{L^2} 
  \ \le\
    C  
\Vert a\Vert_{W^{1, \infty }} \Vert u \Vert_{L^2}\, .
\end{aligned}
$$
 \end{prop}

%%%%%%%%%%%%%%%%%%%%%%%%%%%%%%%%%%%%

%%%%%%%%%%%%%%%%%%%%%%

\subsection{The time dependent case}

In the sequel we  consider functions of $ (t, x) \in [0, T] \times
\RR^n$, considered as functions
of $t $ with values in various  spaces of functions of $x$. In
particular, denote by
$T_a$ the operator acting on $u$ so that for each fixed $t$, 
$ (T_a u)(t) = T_{a(t)} u(t)$. 
\begin{equation}
\label{p5.11}
T_a  u ( t,  x) \, := \, \frac{1}{(2\pi)^n} 
\int e^{ i \xi \cdot x}\, \sigma_a (t, x, \xi) \, \widehat u(\xi) \, d
\xi\, . 
\end{equation}
with 
\begin{equation} 
\sigma_a (t, x, \xi) \, := \, 
\int G (x - y, \xi) \, a(t, y , \xi )\,  dy  
\end{equation}
This definition shows   that   formally
\begin{equation}
\label{comL}
[ \D_t ,  T_a ] \  =\ T_{\D_t a}\,. 
\end{equation}
This yields easy estimates when $\D_t a \in L^\infty$. However, 
if we want to keep the lower bound $s > 1 +\frac{d}{2}$ in  Theorem~\ref{th2.5}
this condition 
need not be satisfied, 
since in the equation \eqref{eq1.1}, $\D_t$ has the 
weight of two spatial derivatives. This is why we 
introduce a slight extension. 

Using the Littlewood-Paley decomposition
\begin{equation}
u = \sum_{k= 0}^{+ \infty} \Delta_k u , \qquad  \mathrm{with} \quad    \widehat {\Delta_k u} := 
\varphi_k \hat u\,,
\end{equation}
as in \eqref{p5.7},  the Besov space $B^{-1, \infty}_\infty$  is defined as the space of 
tempered distributions $u$ such that 
\begin{equation}
\big\|  u   \big\|_{B^{-1, \infty}_\infty} 
\ =\
 \sup_k \  2^{- k} \big\|  \Delta_k u \big\|_{L^\infty}   \ <\
  + \infty. 
\end{equation} 
This space occurs in our analysis because of  the  the following embedding.

\begin{lem}
\label{lememb}
Functions $u \in H^{s}$ belong to  $B^{-1, \infty}_\infty$  when 
$s >  \frac{d}{2} -1$. 
\end{lem}

\noindent
In the spirit of Definition~\ref{def31}, introduce the following notation.

\begin{defi}     
\label{def31a}
For $\mu \in \RR$, let 
  $\Gamma^\mu_{-1} $ denote  the space of  distributions 
$a(x, \xi) $ on $\RR^d \times \RR^{d}  $
which are $C^\infty$ with respect to $\xi$ with values in $B^{-1, \infty}_\infty $
 and such that for all
$\alpha \in \NN^d$ there is a constant $C_\alpha$ such that
\begin{equation}
\label{estimsymb3b}
\forall \xi  \, , \quad  \big\| 
 \partial_{\xi  }^\alpha a ( \cdot , \xi ) \big\Vert_{B^{-1, \infty}_\infty}   \, \le \,
C_\alpha \,
        (1 +   \vert  \xi  \vert )   ^ {  \mu -  \vert
\alpha \vert   }  \, .
\end{equation}

\end{defi}

\begin{defi} [Time dependent symbols]  Let $\mu \in \RR$ and $T > 0$. 
\label{def31b}

i) $\widetilde \Gamma^\mu_0$ denotes the space of locally continuous  functions
$a(t, x, \xi) $ on $[0, T] \times \RR^d \times \RR^{d}  $
which are $C^\infty$ with respect to $\xi$ and such that 
the family $\{ a (t, \ \cdot \, , \, \cdot \,) ; t \in [0, T] \}$ is bounded in 
$\Gamma^\mu_0$. 

ii)  $\widetilde \Gamma^\mu_1$ denotes the space of symbols $a \in \widetilde \Gamma^\mu_0$
such that 

\qquad  -  the family $\{ a (t, \ \cdot \, , \, \cdot \,) ; t \in [0, T] \}$ is bounded in 
$\Gamma^\mu_1$ 

\qquad - the  family $\{ \D_t a (t, \ \cdot \, , \, \cdot \,) ; t \in [0, T] \}$ is bounded in 
$\Gamma^\mu_{-1}$ . 

\end{defi}

For $a \in \widetilde \Gamma^\mu_0$, the operator $T_a$ is defined by \eqref{p5.11}
and the Propositions~\ref{action}, \ref{symbcal}, \ref{paralin} apply for fixed $t$, 
yielding estimates that are uniform in $t$ (see Remark~\ref{rem56}).  
The commutation with $\D_t$ is treated as follows. 

\begin{prop}
\label{commdt}
For $a \in \widetilde \Gamma^\mu_1$ , the commutator 
$[\D_t, T_a] $ maps $C^0([0, T]; H^s)$ to  $C^0([0, T]; H^{s-\mu +1} )$ 
and there is a constant $C$ such that for all $t \in [0, T]$ 
\begin{equation}
\big\| [\D_t , T_a ] u(t) \big\|_{H^{s - \mu - 1} }  \le   C \| u \|_{H^s}.  
\end{equation}
Moreover, the constant $C$ depends only on the supremum for 
$t \in [0, T]$ of a finite number of semi-norms 
\begin{equation}
\label{p5.18} 
\sup_{\xi}  
  (1 +   \vert  \xi  \vert )   ^ {    \vert
\alpha \vert  - \mu  } \big\| 
 \partial_{\xi  }^\alpha a ( \cdot , \xi ) \big\Vert_{B^{-1, \infty}_\infty}    
        \, .
\end{equation}
\end{prop}

\begin{proof} One has,
\begin{equation}
\D_t  \sigma_a (t, \,  \cdot \, , \xi ) \, = \, \sum_{k=0}^{+ \infty}  
\      S_{k- N }(D_x) \big(  \D_t a(t, \, \cdot \,   , \xi ) \big)  \varphi_k (\xi) . 
\end{equation}
where  $S_j$ is the Fourier multiplier with symbol $\chi_j$.  
Since $\D_t a (t, \cdot, \xi)$ belongs to $B^{-1, \infty}_\infty$, 
$$
\big|   S_{k- N }(D_x) \big(  \D_t a(t, \, \cdot \,   , \xi ) \big|  
\ \lesssim \
2^k (1+ | \xi |)^\mu . 
$$
On the support of $\varphi_k$,  that is $| \xi | \approx 2^k$. With similar estimates for the derivatives, 
this shows that  $\D_t a (t, \cdot, \cdot)$ is a bounded family of symbols 
in  $S^{\mu + 1}_{1,1}$. By construction, the spectral property 
that  $\D_t  \hat \sigma_a (t, \eta , \xi )  $  is supported in $|\eta| \le \eps  (1 + | \xi |)$ for 
some $\eps > 0$ is satisfied and therefore the operator 
$(\D_t a ) (t, x, D_x)$ is bounded from $H^s$ to $H^{s - \mu -1}$ for all $s$. 
\end{proof}

%%%%%%%%%%%%%%%%%%%%%%%%%%%%%%%%%%

%%%%%%%%%%%%%%%%%%%%%%%%%%%%%%%%%%
%%%%%%%%%%%%%%%%%%%%%%%%%%%%%%%%%%
%%%%%%%%%%%%%%%%%%%%%%%%%%%%%%%%%%

%========================================


\begin{thebibliography}{99}


\bibitem{Bony} J.M.Bony, 
\textit{ Calcul symbolique et propagation des singularit\'es pour
les \'equations aux d\'eriv\'ees partielles non lin\'eaires},
Ann.Sc.E.N.S. Paris, 14 (1981)209-246.



\bibitem{Co} M. Colin,  and T. Colin,  On a quasilinear
Zakharov system
describing laser-plasma interaction, 
Differential Integral Equations, {\bf 17}(2004)297-330.

\bibitem{co2} M. Colin,  T. Colin, and
G. M\'etivier, Nonlinear models for laser-plasma
interactions, S\'eminaire X-EDP,
Ecole Polyt\'echnique, (2007)X1-X10.


\bibitem{cs}  P. Constantin, and J.-C. Saut,
Local smoothing properties of dispersive
equations,
J. Amer. Math. Soc.
{\bf 1}(1989)413-446.

\bibitem{Hormander} L. H\"ormander.
\textit{Lectures on Nonlinear Hyprbolic Differential Equations},
Math\'ematiques et Applications {\bf  26}, Sringer Verlag, 1997.

\bibitem{kato}  T. Kato,
On the Cauchy problem for the (generalized) Korteweg - de Vries equation. 
Studies in applied mathematics, Adv. Math. Suppl. Stud.
{\bf 8}(1983)93-128.

\bibitem{KPV} C. Kenig, G. Ponce, and L. Vega, 
Small solutions to nonlinear Schr\"odinger equations,
Ann. I.H.P. sec C. {\bf 10} (1993)255-289.

\bibitem{KPV2} C. Kenig, G. Ponce, and L. Vega, 
On the IVP for the nonlinear Schr\"odinger equations,
Contemp. Math. {\bf 189}(1995)353-367.

\bibitem{MetKochel} G. M\'etivier,
\textit{Stability of multidimensional shocks}.
Advances in the theory of shock waves, 25--103,
Progr. Nonlinear Differential Equations
Appl., {\bf 47}, Birkh\"auser, Boston, MA, 2001.

\bibitem{Met2} G. M\'etivier,
{\it Para-Differential Calculus and Applications
to the Cauchy Problem for Nonlinear Systems}
Ennio de Giorgi Math. Res. Center Publ.,
 Edizione della Normale, 2008.

\bibitem{Meyer} Y. Meyer, \textit{Remarques 
sur un th\'eor\`eme de J.M.Bony},
Rend.  Circ. Mat.  Palermo,  Serie II, {\bf 2}(1981)supp. 1, 1-20.
 
 \bibitem{Mizohata} S. Mizohata, 
On some Schr\"odinger type equations. 
Proc. Japan Acad. Ser. A Math. Sci. {\bf 57} no.2.  (1981)81--84. 
 

\bibitem{Mizohata2} S. Mizohata, \textit{On the Cauchy problem}, Notes and Reports in 
Mathematics in Science an Engineering, {\bf 3}, Academic Press, 1985.   

\bibitem{Taylor} M.Taylor.
\textit{Partial Differential Equations}III,
Applied Mathematical Sciences {\bf 117}, Springer, 1996.

\bibitem{tu}  A. Turing,  
The chemical basis of morphogenesis, Phil. Trans. Roy. Soc. B 
{\bf 237}(1952)37-72. 

\end{thebibliography}
\end{document}